\documentclass[12pt,a4paper]{amsart}
\usepackage{graphics}
\usepackage{epsfig}
\usepackage{graphicx}
\theoremstyle{plain}
\usepackage{amssymb}
\usepackage{color}
\advance\hoffset-20mm \advance\textwidth40mm

\newtheorem{theorem}{Theorem}
\newtheorem{lemma}{Lemma}
\newtheorem*{theo*}{Theorem}

\newtheorem{corollary}{Corollary}
\theoremstyle{definition}

\newtheorem*{definition*}{Definition}

\DeclareMathOperator{\Aut}{Aut}

\begin{document}
\sloppy
%

\title[Decomposition  of matrices  from...  ]
{Decomposition  of matrices from $SL_  2(K[x, y])$ }

\author
{Y.Chapovskyi, O.Kozachok, A.Petravchuk}
\address{Institute of Mathematics, National Academy of Sciences of Ukraine,
Tereschenkivska street, 3, 01004 Kyiv, Ukraine}
\email{safemacc@gmail.com}
\address{ Faculty of Mechanics and Mathematics,
Taras Shevchenko National University of Kyiv, 64, Volodymyrska street, 01033  Kyiv, Ukraine}
\email{kozachok.oleksandra@gmail.com}
\address{ Faculty of Mechanics and Mathematics,
	Taras Shevchenko National University of Kyiv, 64, Volodymyrska street, 01033  Kyiv, Ukraine}
\email{ petravchuk@knu.ua, apetrav@gmail.com}

\date{\today}
\keywords{unimodular row; special linear group; generators; decomposition}
\subjclass[2000]{20H05, 20H25, 15A23}

\begin{abstract}
Let $\mathbb{K}$ be an algebraically closed field of characteristic zero and $\mathbb{K}[x,y]$ the polynomial ring. The group $\text{SL}_{2}\left(\mathbb{K}[x,y]\right)$ of all  matrices with  determinant equal to $1$  over $\mathbb{K}[x,y]$ can not be generated by elementary matrices. The known counterexample was pointed out  by P.M. Cohn. Conversely, A.A.Suslin proved that the group $\text{SL}_{r}\left(\mathbb{K}[x_{1},\dots,x_{n}]\right)$ is generated by elementary matrices for $r\ge 3$ and arbitrary $n\geq 2$, the same is true for $n=1$ and arbitrary $r.$ 
It is proven that any  matrix  from $\text{SL}_{2}\left(\mathbb{K}[x,y]\right)$ with at least one entry of degree $\le 2$ is either a product of  elementary matrices or a product of elementary matrices and of a  matrix similar to the one pointed out by P. Cohn. 
For any  matrix   $\begin{pmatrix}\begin{array}{cc} f & g\\ -Q & P \end{array}\end{pmatrix}\in\text{SL}_{2}\left(\mathbb{K}[x,y]\right)$,  we obtain formulas for the homogeneous components $P_i , Q_i$ for  the  unimodular  row $(-Q, P) $ as combinations of homogeneous components of the polynomials $f, g, $ respectively,  with the same coefficients. 

\end{abstract}
\maketitle

\large
\section{Introduction}
Let $\mathbb{K}$ be a field and $A=\mathbb{K}[x_{1},\ldots, x_{n}]$ the polynomial ring in $n$ variables. The  group $GL_{r}(A)$ of all invertible matrices and its subgroup $SL_{r}(A)$ of matrices with  determinant of $1$   was studied by many authors from different points of view (see, for example, \cite{Cohn2}, \cite{Suslin}, \cite{Green},  \cite{Vavilov}, the last paper contains  an extensive literature review). One of the  important questions in studyng $GL_{r}(A)$ (and $SL_{r}(A)$) is the question about generators and relations of these groups. Classical papers of   A.A.Suslin  and P.M.Cohn  answer  this question. In \cite{Suslin}, it was proved  that the group $SL_{r}(A)$ is generated by elementary matrices (or, in other terminology, elementary transvections) if $r\geq 3$ for arbitrary $n\geq 2$; in case $n=1, r\geq 2 $ the proof is elementary. If  $r=2$ and $n\geq 2,$ then the group 
$SL_{r}(A)$ cannot be generated by elementary matrices \cite{Cohn2}. The counterexample from \cite{Cohn2} is the matrix 
$\begin{pmatrix}\begin{array}{cc} x^{2} & xy-1\\ xy+1 & y^{2} \end{array}\end{pmatrix}$ from the group $SL_{2}(\mathbb K[x, y])$. A question arises: how typical is this counterexample? We prove (Theorem 1) that any matrix from $SL_{2}(\mathbb K[x, y])$ with at least one entry of degree $\leq 2$ is either a product of elementary matrices  or a product of a matrix similar to the one pointed out in \cite{Cohn2} and elementary matrices. 

We consider the group $SL_{2}(\mathbb K[x, y])$  over an algebraically closed field $\mathbb K$  of characteristic zero.  Let us recall some definitions and notations. An elementary matrix from the group $SL_{2}(\mathbb K[x, y])$  is of the form $\begin{pmatrix}\begin{array}{cc} 1 & h\\ 0 & 1 \end{array}\end{pmatrix}$ or of the form $\begin{pmatrix}\begin{array}{cc} 1 & 0\\ h & 1 \end{array}\end{pmatrix}$, $h\in\mathbb{K}[x, y].$ A row $(f, g)\in\left(\mathbb{K}[x,y]\right)^{2}$ is called unimodular if there exist polynomials $P, Q\in\mathbb{K}[x,y]$ such that $Pf+Qg=1$ (about some properties of unimodular rows see, for example, \cite{Kumar}).  The latter means that the matrix $\begin{pmatrix}\begin{array}{cc} f & g\\ -Q & P \end{array}\end{pmatrix}$ has the  determinant of $1.$ The unimodular row $(-Q, P)$ will be called associated with the row $(f, g)$. Note that by multiplying  a unimodular row $(f,g)$ from the right  by the matrix $\begin{pmatrix}\begin{array}{cc} 1 & h\\ 0 & 1 \end{array}\end{pmatrix}$, where  $h\in\mathbb{K}[x, y]$, the result is  the unimodular row $(f, g+fh)$. Multiplying   unimodular rows $(f, g)$ by an elementary matrix from the right  defines a linear transformation of the free module of rank 2 over $\mathbb{K}[x,y].$ We call such a transformation an elementary transformation.  The automorphism group 
$ \Aut (\mathbb K[x, y])$ acts naturally on the group $SL_{2}(\mathbb K[x, y])$ by the rule: for any $\theta \in \Aut (\mathbb K[x, y]) $  and $ A=(a_{ij})\in SL_{2}(\mathbb K[x, y])$ put $A^{\theta}=(a_{ij}^{\theta})$ (note that for any elementary matrix $B$ the matrix $B^{\theta} , \theta \in \Aut (\mathbb K[x, y])$ is also an elementary matrix).
We will use unimodular rows, and then the main result will be reformulated in matrix language.   In Lemmas \ref{lem5}, \ref{lem6}, and \ref{lem7} we  prove that any unimodular row $(f, g)$ with $\deg f\leq 2$ is, up to action of an automorphism $\theta\in\Aut\left(\mathbb{K}[x, y]\right),$ one of the forms: $(1, 0) $, or  $(x^{2},\psi (y)x+\gamma)$ with arbitrary $\psi (y)\in \mathbb K[x, y], $  or $(xy-\gamma, x^{k})$, $\gamma\in\mathbb{K}^{*}, k\in \mathbb Z,$ $k\geq 2$.
As a consequence, we obtain the main result ({Theorem \ref{th1}}). 

In the second part of the paper, we consider unimodular rows $(f, g)$, ${\rm deg} f={\rm deg} g=n, n\geq 1.  $ We obtain formulas for homogeneous components $P_i , Q_i$ of  the associated unimodular  row $(-Q, P) $ as combinations of homogeneous components of polynomials $f, g, $ respectively,  with the same coefficients  (Theorem 2). These formulas can be used for studying matrices from $SL_  2(\mathbb K[x, y])$ with entries of any degree.

\section{Some properties of unimodular rows over $\mathbb{K}[x, y]$}
Here   some technical results are collected about unimodular rows (of length $2$) over the polynomial ring $\mathbb{K}[x, y]$.
\begin{lemma}\label{lem1}
	
	{\rm (1)}
	Let $(f, g)$ be a unimodular row over the ring $\mathbb{K}[x, y]$ and $P, Q$  be  polynomials such that $Pf+Qg=1$. Then the rows $(P,Q )$, $(P, g)$ and $(f, Q)$ are also unimodular.
	
	{\rm (2)}
	If $(f, g)$ is a  unimodular row, then for any endomorphism $\theta$ of the ring $\mathbb{K}[x, y],$ the row $(\theta (f), \theta (g))$ is also unimodular.
\end{lemma}

\begin{lemma}\label{lem2}
	Let $f, g, P, Q\in\mathbb{K}[x, y]$ be nonconstant polynomials such that $Pf+Qg=0$. Then there exist polynomials $h_{1}, h_{2}, \varphi, \psi\in\mathbb{K}[x, y]$ such that $f=\varphi h_{1}$, $g=\varphi h_{2}$, $\gcd (h_{1}, h_{2})=1$, $P=\psi h_{2}$, $Q=-\psi h_{1}$.
\end{lemma}
\begin{proof}
	Let $\varphi=\gcd (f, g)$, $\psi_{0}=\gcd (P, Q)$. Then $f=\varphi h_{1}$, $g=\varphi h_{2}$ for coprime polynomials $h_{1}, h_{2}\in\mathbb{K}[x, y]$. Analogously $P=\psi_{0} P_{0}$, $Q=\psi_{0}  Q_{0}$ for some coprime $P_{0}, Q_{0}$. Then by the conditions of the lemma we have
	$$0=Pf+Qg=\varphi\psi_{0}\left(P_{0}h_{1}+Q_{0}h_{2}\right).$$
	The latter equalities imply
	\begin{equation}\label{eq1}
		P_{0}h_{1}+Q_{0}h_{2}=0.
	\end{equation}
	Since $\gcd (h_{1}, h_{2})=1$ we have $P_{0}\mid h_{2}$ and $h_{2}\mid P_{0}$. But the $P_{0}=\alpha h_{2}$ for some $\alpha\in\mathbb{K}^{*}$. Analogously $Q_{0}=\beta h_{1}$ for some $\beta\in\mathbb{K}^{*}$, It follows from \ref{eq1} that $\beta=-\alpha$. Besides, $P=\psi_{0} P_{0}=\psi_{0}\alpha h_{2}$ and $Q_{0}=-\psi_{0}\alpha h_{1}$. Denoting $\psi=\psi_{0}\alpha$ we get $P=\psi h_{2}$, $Q=-\psi h_{1}$.
\end{proof}

\begin{lemma}\label{lem3}
	Let $(f, g)$ be a unimodular row of nonconstant polynomials over $\mathbb{K}[x, y]$ and let $f=f_{0}+\dots +f_{n}$, $g=g_{0}+\dots +g_{l}$ be  decomposition of $f$ and $g,$ respectively, in sums of homogeneous components. Then the polynomials $f_{n}$ and $g_{l}$ are not coprime, i.e., $\deg\gcd (f_{n}, g_{l})\ge 1$.
\end{lemma}
\begin{proof}
	By the conditions of the lemma there exist polynomials $P, Q\in\mathbb{K}[x, y]$ such that $Pf+Qg=1$. Let 
	$$P=P_{0}+\dots +P_{m}, \  Q=Q_{0}+\dots +Q_{k}$$ be their decomposition into sums of homogeneous components. It follows from the equality $Pf+Qg=1$ that $P_{m}f_{n}+Q_{k}g_{l}=0$ (it is obvious that $m+n=k+l$). Assume on the contrary that $\gcd (f_{n}, g_{l})=1$. Then, by Lemma \ref{lem2}, we have 
	\begin{equation}\label{eq2}
		P_{m}=\psi g_{l}, \  Q_{k}=-\psi f_{n}
	\end{equation}
	for a polynomial $\psi\in\mathbb{K}[x, y]$  with $\deg \psi=m-l=k-n$. Denote by $\Omega$ the set of all pairs $(P, Q)$ of polynomials that satisfy the equality $Pf+Qg=1$. Choose a pair $(P, Q)\in\Omega$ such that the sum $m+k=\deg P+\deg Q$ is minimal. Then the equalities (\ref{eq2}) imply that 
	$$\deg (P-\psi g) + \deg (Q+\psi f) <\deg P+\deg Q.$$ Besides, the equality holds: $(P-\psi g)f+(Q+\psi f)g=1$. The latter contradicts the choice of the pair $(P, Q)$. This contradiction shows that $\deg\gcd (f_{n}, g_{l})\ge 1$.
\end{proof}
\begin{corollary}\label{cor1}
	Let $f, g\in\mathbb{K}[x, y]$ be a unimodular row. If $\deg f=1$, then $g=hf+c$ for some $h\in\mathbb{K}[x, y]$, $c\in\mathbb{K}^{*}$.
\end{corollary}
\begin{proof}
	Let $$f=f_{0}+f_{1}, \  g=g_{0}+g_{1}+\dots +g_{l}$$ 
	be the decomposition of polynomials into sums of homogeneous components. Then by {Lemma \ref{lem3}} we have $\deg\gcd (f_{1}, g_{l})\ge 1$. The latter means that $g_{l}$ is divisible by $f_{1}$, i.e. $g_{l}=h_{1} f_{1}$ for some polynomial $h_{1}\in\mathbb{K}[x, y]$. But then the row $(f, g-h_1f)$ is unimodular and $\deg (g-h_1f)<\deg g$. Continuing such considerations we obtain a unimodular row $(f, g-hf)$ for some $h\in\mathbb{K}[x, y]$ such that $\deg (g-hf)=0$, i.e. $g-hf=c$. Obviously $c\ne 0$ and we get $g=hf+c$, $c\in\mathbb{K}^{*}$.
\end{proof}

Let us recall that any quadratic curve $f(x, y)=0$, $\deg f=2$ is reduced by linear transformations of variables to one of the  known  canonical forms. This can be reformulated as follows:

\begin{lemma}\label{lem4}
	Let $f(x, y)\in\mathbb{K}[x, y]$, $\deg f=2$. Then there exist an affine automorphism $\theta$ of  the ring $\mathbb{K}[x, y]$ of the form  $\theta (x)=\alpha_{1}x+\beta_{1}y+\gamma_{1}$, $\theta (y)=\alpha_{2}x+\beta_{2}y+\gamma_{2}$ such that $\theta (f)$ is a polynomial of the following type:
	
	{\rm (1)}
	$f (x, y)=x^{2}+\gamma$, $\gamma\in\mathbb{K}$,
	
	{\rm (2)}
	$f (x, y)=x^{2}+y$,
	
	{\rm (3)}
	$f (x, y)=xy+\gamma$, $\gamma\in\mathbb{K}$.
\end{lemma}

\begin{lemma}\label{lem5}
	Let $(f, g)$ be a unimodular row such that $f=x^{2}+y$. Then this row is reduced to the row $(1, 0)$ by elementary transformations, i.e. there exist elementary matrices $B_{1},\dots ,B_{k}$ such that $(f, g)B_{1},\dots ,B_{k}=(1, 0)$.
\end{lemma}
\begin{proof}
	Let us write the polynomial $g$ as a polynomial of $x$ with coefficients depending on $y$, $$g (x, y)=g_{0}(y)+g_{1}(y)x+\dots +g_{k}(y)x^{k}.$$  Denote 
	$$h (x, y)=g_{2}(y)+g_{3}(y)x+\dots +g_{k}(y)x^{k-2}.$$ Then 
	$$(f, g)\cdot \begin{pmatrix}\begin{array}{cc} 1 & -h\\ 0 & 1 \end{array}\end{pmatrix}=(f, g_{0}(y)+g_{1}(y)x-yh(x, y)).$$
	Note that the polynomial $g^{(1)}=g_{0}(y)+g_{1}(y)x-yh(x, y) $ is of degree $<k $ on $x,$ i.e., $ {\rm deg}_xg<{\rm deg}_xg^{(1)}$.  Repeating this process for the unimodular row $(x^2+y, g^{(1)}) $ 
	we obtain as a result a unimodular row of the form $(x^2+y, g^{(s)}) $ for some $s\geq 2$ with ${\rm deg}_xg^{(s)}\leq 1. $
	So we can assume without loss of generality  that   $g(x, y)=g_{0}(y)+g_{1}(y)x$. By the conditions of the lemma,  there exist polynomials $P (x, y), Q (x, y)\in\mathbb{K}[x, y]$ such that 
	$$P (x, y)(x^{2}+y)+Q (x, y)(g_{0}(y)+g_{1}(y)x)=1.$$ Putting here $y=-x^{2}$ we get the equality 
	$$Q (x, -x^{2})(g_{0}(-x^{2})+xg_{1}(-x^{2}))=1.$$
	It follows from this equality that $g_{0}(-x^{2})+xg_{1}(-x^{2})=c$ for some $c\in\mathbb{K}^{*}$. Since $\deg_{x} g_{0}(-x^{2})$ is even and $\deg_{x} xg_{1}(-x^{2})$ is odd we get $g_{1}=0$ and $g_{0}(y)\in\mathbb{K}$. But then $g=g_{0}\in\mathbb{K}^{*}$ and the unimodular row $(x^{2}+y, g_{0})$ obviously is reduced to the row $(1, 0)$.
\end{proof}

\begin{lemma}\label{lem6} 
	Let $(f, g)$ be a unimodular row, where $f=x^{2}+\gamma$, $\gamma\in\mathbb{K}$. Then this row can be reduced by elementary transformations to either the row $(1, 0)$, or to the row $(x^{2}+\gamma, x\psi(y)+\delta)$, $\delta\in\mathbb{K}$, $\deg \psi (y)\ge 1$.
\end{lemma}
\begin{proof} 
	Write down the polynomial $g (x, y)$ as a polynomial of $x$ with coefficients in $\mathbb K[y] $  
	$$g=g_{0}(y)+g_{1}(y)+\dots +g_{k}(y)x^{k}.$$
	Repeating the consideration from the proof of Lemma \ref{lem5} one can assume without loss of generality that $g=g_{0}(y)+g_{1}(y)x$ for some polynomials $g_{0}(y)$ and $g_{1}(y)$. Since $(x^{2}+\gamma, g_{0}(y)+g_{1}(y)x)$ is a unimodular row,  there exist polynomials $P, Q\in\mathbb{K}$ such that 
	$$P (x^{2}+\gamma)+Q (g_{0}(y)+g_{1}(y)x)=1.$$
	Note that for any polynomial $A(x, y)\in\mathbb{K}[x, y ]$, the polynomials $$\overline{P}(x, y)=P (x, y)+A(x, y)g (x, y), \ \overline{Q}(x, y)=Q (x, y)-A(x, y)(x^{2}+\gamma)$$
	also satisfy the equality $(x^{2}+\gamma)\overline{P}+g(x, y)\overline{Q}=1$. Therefore, without loss of generality, one can reduce the unimodular row $(P, Q)$ by elementary transformations to the row $(P, Q_{0}(y)+Q_{1}(y)x)$ without changing the initial unimodular row $(x^{2}+\gamma, g(x, y))$. We get the equality 
	\begin{equation}\label{eq21}
		P (x, y)(x^{2}+\gamma)+(Q_{0}(y)+Q_{1}(y)x)(g_{0}(y)+g_{1}(y)x)=1.
	\end{equation}
	First, let $\gamma\ne 0$. Substituting in formulas (\ref{eq21}) $x$  for $\sqrt{-\gamma} $ and then $x$ for $-\sqrt{-\gamma}$  we obtain two inclusions  $g_{1}(y)\sqrt{-\gamma}+g_{0}(y)\in\mathbb{K}$ and $-g_{1}(y)\sqrt{-\gamma}+g_{0}(y)\in\mathbb{K}$. It follows from these  inclusions that $g_{0}(y)\in\mathbb{K}$ and $g_{1}(y)\in\mathbb{K}$. But then  from (\ref{eq21}) we see that $Q_{0}(y), Q_{1}(y)\in\mathbb{K}$. The equality (\ref{eq21}) shows also that $g_{1}=0$ and $Q_{1}=0$, i.e., $g(x, y)=c_{1}$ and $Q(x, y)  =c_{2}$ for some $c_{1}, c_{2}\in\mathbb{K}$. Therefore the unimodular row $(x^{2}+\gamma, g)$ can be reduced (by elementary transformations) to the row $(1, 0)$. 
	
	Now let $\gamma=0$, i.e., $f (x, y)=x^{2}$. Putting $x=0$ in the equality (\ref{eq21}) we get $Q_{0}(y)g_{0}(y)=1$. Thus $Q_{0}, g_{0}\in\mathbb{K}^{*}$. The latter means that $g=x\psi (y)+\delta$, where $\psi (y)=Q_{1} (y)$ and $\delta =Q_{0}$. Note that the unimodular row associated with $(x^{2}, x\psi (y)+\delta)$ is the row $(\frac{x\psi (y)-\delta}{\delta^{2}},\delta ^{-2}\psi^{2} (y))$ because the matrix $\begin{pmatrix}\begin{array}{cc} x^{2} & x\psi (y)+\delta\\ \delta^{-2} (x\psi (y)-\delta) & \delta ^{-2}\psi^{2} (y) \end{array}\end{pmatrix}$ has the  determinant  $1.$
\end{proof}

\section{ The main theorem}
We need to consider the last case when the unimodular row is of the form $(xy+\gamma , g(x, y)).$
\begin{lemma}\label{lem7}
	Let $(f, g)$ be a unimodular row with $f(x, y)=xy+\gamma$, $\gamma\in\mathbb{K}$. Then this row can be reduced by elementary transformations to the unimodular row $(x y+\gamma, x^{k})$  or to the row  $(x y+\gamma, (-\gamma^{-1} y)^{k})$ with integer $k\ge 2,$  or 
	to the row $(1, 0)$.
\end{lemma}
\begin{proof}
	By the conditions of the lemma we have an equality of the form
	\begin{equation}\label{eq4}
		P (x, y)(x y+\gamma )+Q(x, y)g(x, y)=1
	\end{equation}
	for some polynomials $P, Q\in\mathbb{K}[x, y]$. Write down the polynomial $g(x, y)$ in the form $g(x, y)=\varphi (x) +\psi (y)+x y h(x, y)$ for some polynomials $\varphi (x), \psi (y), h(x, y)\in\mathbb{K}[x, y]$. Then we get the equality
	$$(x y+\gamma, g)\begin{pmatrix}\begin{array}{cc} 1 & -h(x, y)\\ 0 & 1 \end{array}\end{pmatrix} = (x y+\gamma, \varphi(x)+\psi(y)-\gamma h(x, y)).$$
	If $h (x, y)\ne 0$ we can write $h (x, y)=\varphi_{1}(x)+\psi_{1}(y)+x y h_{1}(x, y)$ and repeat the previous considerations. As a result,  we may assume without loss of generality that $g (x, y)=\varphi (x)+\psi (y)$. Analogously repeating considerations from the proof of Lemma \ref{lem6} we may assume that $Q (x, y)=u (x) + v (y)$ for some polynomials $u (x), v (y)\in\mathbb{K}[x, y]$.
	
	First, let $\gamma\ne 0$. Let us put $y=-\gamma /x$ in the equality (\ref{eq4}). We get $(u(x)+ v (-\gamma /{x}))(\varphi (x)+\psi (-\gamma /{x}))=1$.
	One can easily prove that an element $p (x, x^{-1})$ from ring $\mathbb{K}[x, x^{-1}]$ is invertible in this ring if and only if $p=\alpha x^{k}$ for some $k\in\mathbb{Z}$, $\alpha\in\mathbb{K}^{*}$. So, we have $g (x, y)=x^{k}$, $Q (x, y)=(-\gamma^{-1} y)^{k}$ or $g  (x, y)=(-\gamma^{-1} y)^{k}$, $Q (x, y)=x^{k}$ for some $k\ge 0$. In any case, the polynomial $P (x, y)$ is of the form 
	$$P(x, y)=\frac{1-(-\gamma^{-1} x y)^k}{\gamma+x y}=\gamma^{-1}\left(1+\left(-\frac{x y}{\gamma}\right)+\dots +\left(-\frac{x y}{\gamma}\right)^{k-1}\right).$$
	As a result, we get two unimodular rows:
	
	1) $(x y+\gamma, x^{k})$ with the associated row $\left( -(-\gamma^{-1} y)^{k},     \frac{1-(-\gamma^{-1} x y)^{k}}{\gamma+x y}, \right)$;
	
	2) $(x y+\gamma, (-\gamma^{-1} y)^{k})$ with the associated row $\left( -x^{k},  \frac{1-(-\gamma^{-1} x y)^k}{\gamma+x y}, \right)$.
	
	Note that one can assume that $k\geq 2 $. Really, in other case the row $(x y+\gamma, x^{k})$ is reduced to the row $(1, 0)$  because of Corollary \ref{cor1}. 
	Let now $\gamma =0$. Let us replace $x$ with 0 in the equality (\ref{eq4}). Then we have $(u(0)+v(y))(\varphi (0)+\psi (y))=1$. This equality implies obviously $v(y), \psi (y)\in\mathbb{K}$. Analogously after substituting 0 instead of $y$ in (\ref{eq4}) we get $v(x), \psi (x)\in\mathbb{K}$. We see that in this case the polynomial $g(x, y)$ is constant  and therefore the  unimodular row can be reduced to the row $(1, 0)$. The proof is complete.
\end{proof}

\begin{theorem}\label{th1}
	Let $A=\begin{pmatrix}\begin{array}{cc} a_{1\,1}(x, y) & a_{1\,2}(x, y)\\ a_{2\,1}(x, y) & a_{2\,2}(x, y) \end{array}\end{pmatrix}\in\text{SL}_{2}\left(\mathbb{K}[x, y]\right)$. If $\deg a_{i\,j}=2$ for some $i,j\in \{1, 2\}, $ then there exists an automorphism $\theta\in\Aut\left(\mathbb{K}[x, y]\right)$ such that $A^{\theta}$ is one of the types:
	
	1) $A^{\theta}=B_{1} B_{2}\dots B_{k}$, $k\ge 1$, $B_{i}$ are elementary matrices;
	
	2) $A^{\theta}=B_{1}\dots B_{s} C B_{s+1}\dots B_{k}$, where $B_{1}\dots B_{s}$, $B_{s+1}\dots B_{k}$ are elementary matrices and $C$ is one of the form:
	
	a) $\begin{pmatrix}\begin{array}{cc} x^{2} & x\psi(y)+\delta\\ 
			\frac{x\psi(y)-\delta}{\delta^{2}} & \frac{\psi(y)^{2}}{\delta ^{2}} \end{array}\end{pmatrix}$  b) $\begin{pmatrix}\begin{array}{cc} x y+\gamma &  x^k \\
			-(-\gamma^{-1} y)^{k}  &  \frac{1-(-\gamma^{-1} x y)^k}{\gamma+x y} \end{array}\end{pmatrix}$
	
	for some $\delta , \gamma\in\mathbb{K}^{*}$, $\psi (y)\in\mathbb{K}[x, y]$, $k\in\mathbb{Z} $. 
\end{theorem}
\begin{proof}
	Multiplying the matrix $A$ from the left or from the right by the matrix $\begin{pmatrix}\begin{array}{cc} 0 & 1\\ -1 & 0 \end{array}\end{pmatrix}$ we can assume without loss of generality that $i=j=1$, i.e. $\deg a_{1\,1}\le 2$. Applying a linear automorphism $\theta$ to the matrix $A$  we can reduce (by Lemma \ref{lem3}) the element $a_{1\,1} (x, y)$ to one of the forms
	
	1) $a_{1\,1} (x, y)=x^{2}+y$, 
	
	2) $a_{1\,1} (x, y)=x^{2}+\gamma$,
	
	3) $a_{1\,1} (x, y)=x y+\gamma$.
	
	First, let $a_{1\,1} (x, y)=x^{2}+y$. Then applying Lemma \ref{lem5} to the first row of the matrix $A$ we get the matrix $\begin{pmatrix}\begin{array}{cc} 1 & 0\\ b(x, y) & 1 \end{array}\end{pmatrix}$ for a polynomial $b(x, y)\in \mathbb K[x, y]$, recall that multiplying from the left by elementary matrices makes elementary transformations in the first and second rows of $A$.
	The latter means that $A$ is a product of elementary matrices.
	In the case $a_{1\,1} (x, y)=x^{2}+\gamma$, $\gamma\in\mathbb{K}$ we get either a product of elementary matrices $A=B_{1}\dots B_{k}$ or a product of the form $A=B_{1}\dots B_{i-1} C B_{i+1}\dots B_{k}$, where $B_{i}$ are elementary matrices and $C$ is of the form
	$$C=\begin{pmatrix}\begin{array}{cc} x^{2} & x\psi(y)+\delta\\ 
			\delta^{-2}(x\psi(y)-\delta) & \delta ^{-2}\psi(y)^{2} \end{array}\end{pmatrix}.$$
	By Lemma \ref{lem7}, the last case $a_{1\,1} (x, y)=x y+\gamma$, $\gamma\in\mathbb{K}$  yields the product $A=B_{1}\dots B_{i-1} C B_{i+1}\dots B_{k}$ with $C$ of the form 
	$$C_1=\begin{pmatrix}\begin{array}{cc} x y+\gamma & x^{k}\\ 
			-(-\gamma^{-1} y)^{k}   &   \frac{1-(-\gamma^{-1} x y)^k}{\gamma+x y} \end{array}\end{pmatrix} $$
	or of the form
	$$C_2=\begin{pmatrix}\begin{array}{cc} x y+\gamma & (-\gamma^{-1} y)^{k} \\ 
			-x^{k}  &   \frac{1-(-\gamma^{-1} x y)^k}{\gamma+x y} \end{array}\end{pmatrix} $$
	
	Note that the matrices
	$C_{1}$  and $C_{2}$  are conjugated by the automorphism $\theta : x\mapsto -\gamma ^{-1}y,\ y\mapsto -\gamma x$. The proof is complete.
\end{proof}

\section{Formulas for  associated rows}

If a unimodular row $(f, g)$ is given, then there exists a unimodular row $(-Q, P)$ such that $P f+Q g=1$ (then the matrix $\begin{pmatrix}\begin{array}{cc} f & g\\ -Q & P \end{array}\end{pmatrix}$ has  determinant of 1).
Such a row $(-Q, P)$ is unique up to a row $(-\lambda Q, \lambda P)$ for an arbitrary polynomial $\lambda\in\mathbb{K}[x, y]$.
Really, if $P' f+Q' g=1$ for a row $(P', Q')$, then $(P-P') f+(Q-Q') g=0$. By Lemma \ref{lem2},
$P-P' =\lambda g$, $Q-Q' =\lambda f$ for some $\lambda\in\mathbb{K}[x, y]$ and therefore
$$(P', Q')=(P, Q)+(-\lambda g, \lambda f).$$
Let us point out how one can write  homogeneous components of polynomials $P$, $Q$ using  homogeneous components of $g$ and $f$ respectively. We restrict ourselves only to polynomials $f$, $g$ of the same degree. Let
$\deg f=\deg g=n$. Then obviously $\deg P=\deg Q=m$ for some $m$. Write down polynomials $f , g, P, Q$ as sums of their homogeneous components
$$f=f_{0}+\dots +f_{n}, g=g_{0}+\dots +g_{n},$$
$$P=P_{0}+\dots +P_{m}, Q=Q_{0}+\dots +Q_{m}.$$
Denote $\varphi = \gcd (f_{n}, g_{n})$.
We assume that all the polynomials $f$, $g$, $P$, $Q$ are  nonconstant ones. Then by the Lemma \ref{lem3}, $\deg \varphi \ge 1.$ It turns out that  $\varphi^{i+1} P_{m-i}$ and $\varphi^{i+1} Q_{m-i}$ can be written as linear combinations of $ g_i's$  and $f_i's $, respectively, with the same polynomial coefficients.
\begin{theorem}\label{th2}
	There exist homogeneous polynomials $\alpha_{0}, \dots , \alpha_{m}$ such that for $0\le i\le m$
	\begin{equation} \label{eq:star2}
		\begin{split}
			\varphi^{i+1} P_{m-i} &=\sum_{j=0}^{\min (i, n)} \varphi^{j} \alpha_{i-j} g_{n-j}, \\
			-\varphi^{i+1} Q_{m-i} &=\sum_{j=0}^{\min (i, n)} \varphi^{j} \alpha_{i-j} f_{n-j}.
		\end{split}\tag{$\star$}
	\end{equation}
\end{theorem}
\begin{proof}
	Induction on $i$. The case $i=0$ is a consequence of Lemma \ref{lem2}. Really, we have
	$P_{m} f_{n} + Q_{m} g_{n}=0$. Let 
	$$\varphi =\gcd (f_{n}, g_{n}), \  h_{1}=f_{n}/\varphi, \ h_{2}=g_{n}/\varphi .$$ By  Lemma \ref{lem2} $P_{m} =\psi h_{2}$, $Q_{m} =-\psi h_{1}$ for some $\psi\in\mathbb{K}[x, y]$. Then 
	$$\varphi P_{m}=\psi\varphi h_{2}=\psi g_{n}, -\varphi Q_{m}=\psi\varphi h_{1}=\psi f_{n}.$$
	Putting $\alpha_{0}=\psi$ we get the case $i=0$. Let the formulas \eqref{eq:star2} be true for $i' < i$, let us prove it for $i.$. Since $P f + Q g=1$ we have equalities for homogeneous components in the left side of the later equality: $(P f +  Q g)_{m+n-i}=0$ for $0\leq i\leq  m$. But the left side of the latter equality can be written in the form
	$$\sum_{k=0}^{\min (i, n)} \left(P_{m-i+k} f_{n-k}+Q_{m-i+k} g_{n-k}\right)=0.$$
	After multiplying this equality by $\varphi^{i+1}$ we can rewrite it for $0\leq i\leq m$ in the form 
	$$\sum_{k=0}^{\min (i, n)} \varphi^{k}\left(\varphi^{i-k+1} P_{m-i+k} f_{n-k} + \varphi^{i-k+1} Q_{m-i+k} g_{n-k}\right)=0.$$
	Replacing $P_{m-i+k}$ and $Q_{m-i+k}$, $k\ge 1$ by their expressions due to the induction hypothesis we obtain the equality (we denote $min (i, n)$ by $i\wedge n $  for brevity in the next part of the proof):
	\begin{multline*}
		0=\varphi^{i+1}\left(P_{m-i}f_{n}+Q_{m-i}g_{n}\right)+\\
		+\sum_{k=1}^{i\wedge n}\varphi^{k}\left(f_{n-k}\sum_{j=0}^{i\wedge n}\varphi^{j}\alpha_{i-k-j}g_{n-j}-g_{n-k}\sum_{j=0}^{i\wedge n}\varphi^{j}\alpha_{i-k-j}f_{n-j}\right).
	\end{multline*}
	The last equality can be rewritten in the form
	\begin{multline*}
		\varphi^{i+1}\left(P_{m-i}f_{n}+Q_{m-i}g_{n}\right)+g_{n}\sum_{k=1}^{i\wedge n}\varphi^{k}\alpha_{i-k}f_{n-k}-f_{n}\sum_{k=1}^{i\wedge n}\varphi^{k}\alpha_{i-k}g_{n-k}+ \\
		+\underset{\substack{1\le j,k\le n\\ j+k\le i}}{\sum}\varphi^{j+k}\alpha_{i-k-j}f_{n-k}g_{n-j}-
		\underset{\substack{1\le j,k\le n\\j+k\le i}}{\sum}\varphi^{j+k}\alpha_{i-k-j}f_{n-j}g_{n-k}=0.
	\end{multline*}
	Note that the last two sums in this equality give as result 0 and we can write the last equality as 
	\begin{multline*}
		\left(\varphi^{i+1}P_{m-i}-\sum_{k=1}^{i\wedge n}\varphi^{k}\alpha_{i-k}g_{n-k}\right)f_{n}+\\
		+\left(\varphi^{i+1}Q_{m-i}-\sum_{k=1}^{i\wedge n}\varphi^{k}\alpha_{i-k}f_{n-k}\right)g_{n}=0.
	\end{multline*}
	It follows from Lemma \ref{lem2} that there exists a polynomial $\alpha_{i}$ such that
	$$\varphi^{i+1}P_{m-i}-\sum_{k=1}^{i\wedge n}\varphi^{k}\alpha_{i-k}g_{n-k}=\alpha_{i}g_{n},$$ $$\varphi^{i+1}Q_{m-i}+\sum_{k=1}^{i\wedge n}\varphi^{k}\alpha_{i-k}f_{n-k}=- \alpha_{i}f_{n}.$$
	These equalities can be rewritten (in the initial notation) in the form
	$$\varphi^{i+1}P_{m-i}=\sum_{k=0}^{\min(i,n)}\varphi^{k}\alpha_{i-k}g_{n-k},\qquad-\varphi^{i+1}Q_{m-i}=\sum_{k=0}^{\min(i,n)}\varphi^{k}\alpha_{i-k}f_{n-k}.$$
	The proof is complete.
\end{proof} 


%
\end{document}